\documentclass[psamsfonts]{amsart}

\usepackage{fullpage}
\usepackage{float}

\usepackage[parfill]{parskip}
\usepackage{placeins}

\usepackage{cite}

\usepackage[colorlinks=true,pagebackref]{hyperref}
\hypersetup{colorlinks=false, pdfborder={0 0 0}}
\usepackage{amsfonts,amssymb,amscd,amsmath,enumerate,verbatim,}
\usepackage{amssymb}

\usepackage[all]{xy}
\SelectTips{eu}{} 
\xyoption{curve}

\usepackage{tikz}
\usetikzlibrary{matrix,arrows,decorations.pathmorphing, chains, positioning,scopes}

\usepackage{young}
\usepackage[vcentermath]{youngtab}
\usepackage{eufrak}

\newcommand{\scrptQ}{\mathcal Q}
\newcommand{\scrptR}{\mathcal R}

\newcommand{\scrptV}{\mathcal V}

\newcommand{\bbP}{\mathbb P}

\newcommand{\bbS}{\mathbb S}

\newcommand{\bbZ}{\mathbb Z}
\def\into{{\rightarrowtail}}
\def\onto{\twoheadrightarrow}

\newcommand{\pseudofrac}[3]{%
  \begin{subarray}{l}#1\\#2\\#3\end{subarray}%
}
\DeclareMathOperator{\Hom}{Hom}

\newcommand{\niso}{\cong}

\DeclareMathOperator{\Ext}{Ext}
\DeclareMathOperator{\Tor}{Tor}

\DeclareMathOperator{\im}{im}
\DeclareMathOperator{\drsum}{\oplus}
\DeclareMathOperator{\Drsum}{\bigoplus}

\DeclareMathOperator{\tnsr}{\otimes}
\DeclareMathOperator{\TnsrAlg}{\bigotimes}

\DeclareMathOperator{\ExtAlg}{\bigwedge}
\DeclareMathOperator{\wdg}{\wedge}
\DeclareMathOperator{\Sym}{Sym}
\DeclareMathOperator{\GL}{GL}

\DeclareMathOperator{\fmod}{mod}

\DeclareMathOperator{\Proj}{Proj}
\theoremstyle{definition}
\newtheorem{defn}{Definition}[section]
\newtheorem{definition}[defn]{Definition}

\newtheorem{example}[defn]{Example}

\theoremstyle{plain}

\newtheorem{theorem}[defn]{Theorem}
\newtheorem*{theorem*}{Theorem}

\newtheorem{lemma}[defn]{Lemma}

\newtheorem{claim}[defn]{Claim}

\theoremstyle{remark}

\begin{document}

\title[Elementary Equivariant Modules]{Elementary Equivariant Modules}

%\dedicatory{Dedicated to...}

\author[Author]{Mikhail Gudim} 
\address{} \email{}


\date{\today} 

\maketitle

\begin{abstract}
We study equivariant modules over $\GL(V)$ over the polynomial ring $R=\Sym V$. We introduce for every partition $\lambda$ the elementary equivariant module $M_{\lambda}$. Then we prove that any finitely generated equivariant module admits a filtration with associated graded being the direct sum of modules of only two kinds: either $M_{\lambda}$ or truncations of $M_{\lambda}$. We show that each $M_{\lambda}$ has a linear resolution and describe also the resolution of its truncations.  
\end{abstract}

\makeatletter
\@starttoc{toc}
\makeatother

\section{Introduction}

We will work in the category of finitely generated $\GL(V)$-equivariant modules over the polynomial ring (see Section ~\ref{sec:CatMod} for definitions). This and related categories are the subject of \cite{SamSnowd} where a much more elaborate treatment is given. In our paper we concentrate on studying the resolutions of what we will call (see Section~\ref{sec:Elem}) "elementary equivariant modules associated with partition $\lambda$" and denote by $M_{\lambda}$. To our knowledge, these modules first appeared in Section 2.2 of \cite{SamSnowd} where their properties are discussed. One thing that makes $M_{\lambda}$'s remarkable is that every equivariant $\GL(V)$-module has a filtration with the associated graded being the direct sum of either $M_{\lambda}$'s or their truncations. This is our Theorem~\ref{thm:filt}, but this was already known (see Corollary 2.2.7 of \cite{SamSnowd}). It turns out that the resolutions of $M_{\lambda}$'s are linear and their terms are described in Theorem~\ref{thm:main}. As Steven Sam explained to me, this result can be obtained (see Section~\ref{subsec:Geom}) using the so-called geometric technique of Weyman (see Theorem 5.1.2 of \cite{Wey}). However, our method is completely algebraic and it also allows to study the resolutions of truncations of $M_{\lambda}$ which we do in Section~\ref{sec:ResTr}. Our approach is to analyze the complex which computes $\Tor^R_{\bullet}(M_{\lambda}, \Bbbk)$ in explicit way. The main difficulty in this analysis is to establish bijection between irreducible summands of tensor product of three representations induced by associativity. In general, this is a pretty hard problem (see Section 5 of \cite{Knu-Tao-Wood}), but we only need a very special case of it. This is done in Section~\ref{sec:Emb}.

I would like to thank Steve Sam, who pointed out Lemma~\ref{lem:Sam} to me and read a draft of this paper. Also I thank Ragnar-Olaf Buchweitz for helpful suggestions.

\subsection*{Notation and Conventions} Let $\Bbbk$ be an algebraically closed field of characteristic zero and $V$ be a vector space over $\Bbbk$ of dimension $n$ with basis $x_1, \cdots x_n$. We will think of elements of $V$ as having degree $1$ and elements of $\Bbbk$ have degree $0$ of course. We will only work with finitelly generated graded modules and graded maps. Such category will be denoted by $S-\fmod$ where any graded ring can take place of $S$.

Let $\lambda=(\lambda_1, \lambda_2, \cdots \lambda_m)$ be some partition. By $\ExtAlg^{\lambda}V$ we will mean the tensor product $\ExtAlg^{\lambda_1}V \tnsr \ExtAlg^{\lambda_2}V \tnsr \cdots \ExtAlg^{\lambda_m}V$ and similarly for the symmetric powers. We will denote the Schur functor corresponding to $\lambda$ by $\bbS_{\lambda} V$ and most often we will abbreviate it to simply $\bbS_{\lambda}$. 

\section{Technical Results}
\label{sec:TechRes}

\subsection{Coassociativity of the Exterior Algebra}

See Section I of \cite{Ak-Buch-Wey} for details of this section. The exterior algebra $\ExtAlg V$ is a graded-commutative Hopf algebra. Let $l:=(l_1, \cdots l_t)$, $a:=(a_1, \cdots a_t)$ and $b:=(b_1, \cdots b_t)$ be vectors with all integer components. Consider the map

\[
\Phi(l,a) \colon \ExtAlg^{l}V= \ExtAlg^{l_1}V \tnsr \cdots \ExtAlg^{l_t}V \to (\ExtAlg^{l_1-a_1}V \tnsr \cdots \ExtAlg^{l_t-a_t}V) \tnsr (\ExtAlg^{a_1}V \tnsr \cdots \ExtAlg^{a_t}V)=\ExtAlg^{l-a}V \tnsr \ExtAlg^{a}V 
\]

The map $\Phi(l,a)$ is a tensor product of the appropriate components of comultiplications followed by a permutation of factors.

As a consequence of coassociativity of $\ExtAlg V$ the following diagram commutes:

\begin{figure}[h!]

\begin{tikzpicture}
\matrix(m) [matrix of math nodes, 
row sep=4.0em, column sep=8.0em, 
text height=1.5ex, text depth=0.25ex]
{\ExtAlg^l V & \ExtAlg^{l-b}V \tnsr \ExtAlg^b V\\
 & \ExtAlg^{l-b-a}V \tnsr \ExtAlg^{a}V \tnsr \ExtAlg^bV \\
\ExtAlg^{l-a}V \tnsr \ExtAlg^aV & \ExtAlg^{l-a-b}V \tnsr \ExtAlg^bV \tnsr \ExtAlg^aV \\};

\path[->,font=\scriptsize] 
(m-1-1) edge node[above]{$\Phi(l,b)$}(m-1-2)
(m-3-1) edge node[above]{$\Phi(l-a,b) \tnsr 1_{\ExtAlg^aV}$}(m-3-2)
(m-1-1) edge node[right]{$\Phi(l,a)$}(m-3-1)
(m-1-2) edge node[right]{$\Phi(l-b,a) \tnsr 1_{\ExtAlg^bV}$} (m-2-2) 
(m-3-2) edge node[right]{$\sigma$} node[left]{$\niso$}(m-2-2);
\end{tikzpicture}
\caption{}
\label{fig:coass}
\end{figure}
\FloatBarrier

The map $\sigma$ is a permutation of factors.

\subsection{Representations of $\GL(V)$}
\label{subsec:Rep}

Every finite-dimensional polynomial (more generally, rational) representation of $G:=\GL(V)$ decomposes into direct sum of irreducible ones. Every irreducible polynomial representation of $G$ is isomorphic to a Schur Functor. See Section 8.2 of ~\cite{Fult} for a reference.

Let $\lambda:=(\lambda_1, \cdots \lambda_n)$ be a partition. By $D(\lambda)$ we will mean the Young diagram of $\lambda$ which we draw as a set of boxes in the plane. For example, the Young diagram of the partition $(3,2,2,1)$ is

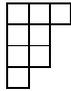
\[
{\scriptsize\young(\hfil \hfil \hfil,\hfil \hfil,\hfil \hfil,\hfil)}
\]

The conjugate partition of $\lambda$ is the partition $\tilde{\lambda}=(\tilde{\lambda}_1, \cdots \tilde{\lambda}_t)$ where $\tilde{\lambda}_i$ is the number of boxes in the $i$-th column of $D(\lambda)$. For example, the conjugate partition of $(3,2,2,1)$ is the partition $(4,3,1.0)$ and its Young diagram is:

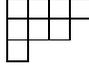
\[
{\scriptsize\young(\hfil \hfil \hfil \hfil,\hfil \hfil \hfil,\hfil)}
\]

Let us recall:

\begin{theorem}
There is only one representation occurring in both $\Sym_{\lambda} V$ and $\ExtAlg^{\tilde{\lambda}} V$. This common representation is the Schur functor $\bbS_{\lambda}$ and can be realized as 

\begin{enumerate}
\item As an image of the (unique) map $\Sym_{\lambda} V \to \underset{(i,j) \in D(\lambda)}{\TnsrAlg} V(i,j) \to \ExtAlg^{\tilde{\lambda}} V$ and thus as a quotient of $\Sym_{\lambda} V$, or

\item as an image of the (unique) map $\ExtAlg^{\tilde{\lambda}} V \to \underset{(i,j) \in D(\lambda)}{\TnsrAlg} V(i,j) \to \Sym_{\lambda} V$ and thus as a quotient of $\ExtAlg^{\tilde{\lambda}} V$
\end{enumerate}

Here $V(i,j)$ means the copy of $V$ indexed by the box $(i,j)$ of $D(\lambda)$. The first map in (1) is the tensor product of components of comultiplications in the symmetric algebra: $\Sym_{\lambda_i}V \to \underset{1 \leq j \leq \lambda_i}{\TnsrAlg}V(i,j)$ and the second map is the tensor product of projections $\underset{1 \leq i \leq \tilde{\lambda}_j}{\TnsrAlg}V(i,j) \to \ExtAlg^{\tilde{\lambda}_j}V$
\end{theorem}

For each semistandard tableaux $T$ one can write down an element $v_T$ of $\ExtAlg^{\tilde{\lambda}}V$ which is a tensor product of wedge products of elements, the $i$-th factor being the wedge product of elements (in order) occuring in the $i$-th column of $T$. These elements are linearly independent and span the subspace of $\ExtAlg^{\tilde{\lambda}}V$ isomorphic to $\bbS_{\lambda}$. Moreover, the images of those elements under the composite map in (2) above form a basis for the subspace of $\Sym_{\lambda}V$ isomorphic to  $\bbS_{\lambda}$.

According to our grading conventions, $\bbS_{\lambda}$ is a graded vector space concentrated in the degree $|\lambda|$- the number of boxes in $D(\lambda)$. For example, both $\bbS_{(3,2,2,1)}$ and $\bbS_{(4,3,1.0)}$ are of degree $3+2+2+1=4+3+1+0=8$.

\subsection{Pieri Inclusions}

For a partition $\lambda$ define $VS(\lambda, k)$ (here "$VS$"  stands for "vertical strip") to be the set of all partitions obtained from $\lambda$ by adjoining $k$ boxes, no two in the same row. Similarly let $HS(\lambda, k)$ (here "$HS$" stands for "horizontal strip") to be the set of all partitions obtained from $\lambda$ by adjoining $k$ boxes with no two in the same column. Let us recall the well-known Pieri Formulas (see Section 6.1 of \cite{Fult-Harr}):

\[
\bbS_{\lambda} \tnsr \Sym_k V \niso \underset{\eta \in HS(\lambda, k)}{\Drsum} \bbS_{\eta}
\] 

and

\[
\bbS_{\lambda} \tnsr \ExtAlg^k V \niso \underset{\eta \in VS(\lambda, k)}{\Drsum} \bbS_{\eta}
\] 

For a fixed $\lambda$ and $k$ let us take $\eta \in VS(\lambda, k)$. Suppose $a_j$ boxes were added to the $j$-th column of $\lambda$ to get $\eta$ (in particular, the $a_j$'s add up to $k$) Let $a:=(a_1, \cdots a_j)$. The inclusion $\bbS_{\eta} \into \bbS_{\lambda} \tnsr \ExtAlg^k V$ is the composition (see proof of Theorem IV.2.1 in \cite{Ak-Buch-Wey})

\[
\bbS_{\eta}  \into \ExtAlg^{\tilde{\eta}}V \to \ExtAlg^{\tilde{\lambda}}V \tnsr (\ExtAlg^{a} V) \to \ExtAlg^{\tilde{\lambda}}V \tnsr \ExtAlg^k V \to \bbS_{\lambda} \tnsr \ExtAlg^kV
\]

\begin{example}
\label{ex:Pieri1}
The formula for the Pieri inclusion $\bbS_{(2,2)}V \into \Sym_2V \tnsr \Sym_2V$ is:

\begin{eqnarray*}
{\scriptsize\young(ab,cd)} \mapsto \pseudofrac{a}{\wdg}{c} \tnsr \pseudofrac{b}{\wdg}{d} \mapsto (\pseudofrac{a}{\tnsr}{c} - \pseudofrac{c}{\tnsr}{a}) \tnsr (\pseudofrac{b}{\tnsr}{d} - \pseudofrac{d}{\tnsr}{b}) \mapsto \\
\pseudofrac{a}{\tnsr}{c} \tnsr \pseudofrac{b}{\tnsr}{d}-\pseudofrac{a}{\tnsr}{c} \tnsr \pseudofrac{d}{\tnsr}{b}-\pseudofrac{c}{\tnsr}{a} \tnsr \pseudofrac{b}{\tnsr}{d}+\pseudofrac{c}{\tnsr}{a} \tnsr \pseudofrac{d}{\tnsr}{b} \mapsto \\
ab \tnsr cd-ad \tnsr cb-cb \tnsr ad+cd \tnsr ab
\end{eqnarray*}

\end{example}

\begin{example}
\label{ex:Pieri2}
The formula for the Pieri inclusion $\bbS_{(2,1)}V \into V \tnsr \Sym_2V$ is:

\begin{eqnarray*}
{\scriptsize\young(ab,c)} \mapsto \pseudofrac{a}{\wdg}{c} \tnsr b \mapsto \pseudofrac{a}{\tnsr}{c} \tnsr b - \pseudofrac{c}{\tnsr}{a} \tnsr b \mapsto c \tnsr ab - a \tnsr cb
\end{eqnarray*}

\end{example}

\subsection{Embeddings into $V \tnsr \bbS_{\lambda} \tnsr \ExtAlg^kV$}
\label{sec:Emb}

We will need a technical result which describes the embeddings of irreducible representations into $V \tnsr \bbS_{\lambda} \tnsr \ExtAlg^k V$. First let us think about this tensor product with brackets placed as follows: $V \tnsr (\bbS_{\lambda} \tnsr \ExtAlg^k V)$. With this placement of brackets we will think about each irreducible summand $\eta$ as being obtained from the diagram of $\lambda$ by first adding $k$ boxes according to Pieri rule (i.e. first tensor with $\ExtAlg^k V$) and then adding one more box (tensor with $V$). Moreover, we will mark the $k$ boxes that were added first with the symbol "$\ExtAlg$" and the box which was added last will be marked with the symbol "$V$". 

\begin{example}

Suppose the dimension of $V$ is $3$. Consider the tensor product $V \tnsr \bbS_{(2,1,0)} \tnsr \ExtAlg^2 V$. The irreducible summands of $\bbS_{(2,1,0)} \tnsr \ExtAlg^2 V$ are obtained from the diagram of $\lambda$ by adding two boxes according to the Pieri rule. Let us mark these boxes with the symbol '$\ExtAlg$'. Thus the summands of $\bbS_{(2,1,0)} \tnsr \ExtAlg^2 V$ are:

{\scriptsize\young( \hfil \hfil \ExtAlg,\hfil \ExtAlg)}, {\scriptsize\young( \hfil \hfil \ExtAlg,\hfil,\ExtAlg)}, {\scriptsize\young( \hfil \hfil,\hfil \ExtAlg,\ExtAlg)}

Now to get all the summands of $V \tnsr (\bbS_{(2,1,0)} \tnsr \ExtAlg^2 V)$ we add one more box to all the summands above. Let us mark this box with the symbol "$V$". We get:

{\scriptsize\young( \hfil \hfil \ExtAlg V,\hfil \ExtAlg)}, {\scriptsize\young( \hfil \hfil \ExtAlg,\hfil \ExtAlg V)}, {\scriptsize\young( \hfil \hfil \ExtAlg,\hfil \ExtAlg,V)},\\

{\scriptsize\young( \hfil \hfil \ExtAlg V,\hfil,\ExtAlg)}, {\scriptsize\young( \hfil \hfil \ExtAlg,\hfil V,\ExtAlg)}, {\scriptsize\young( \hfil \hfil \ExtAlg,\hfil,\ExtAlg,V)},\\

{\scriptsize\young( \hfil \hfil V,\hfil \ExtAlg,\ExtAlg)}, {\scriptsize\young( \hfil \hfil,\hfil \ExtAlg,\ExtAlg V)}, {\scriptsize\young( \hfil \hfil,\hfil \ExtAlg,\ExtAlg,V)}.
\end{example}

Note that with this bracket placement no "$\ExtAlg$"'s will be to the right or below the "$V$". Each such (labeled) diagram $\eta$ defines an embedding $\bbS_{\eta} \into V \tnsr (\bbS_{\lambda}  \tnsr \ExtAlg^k V)$ in the following way: suppose $\eta$ has $a_j$ $\ExtAlg$'s in the column $j$ and the $V$ is in the column $s$ Let $a:=(a_1, \cdots a_t)$ and $b$  be the vector all of whose components are zero, except $1$ in the position $s$.

\begin{eqnarray*}
\bbS_{\eta} \into \ExtAlg^{\tilde{\eta}} V \xrightarrow{\Phi(\tilde{\eta}, b)} V \tnsr \ExtAlg^{\tilde{\eta}-b}V \xrightarrow{1_V \tnsr \Phi(\tilde{\eta}-b,a)} V \tnsr (\ExtAlg^{\tilde{\eta}-b-a}V \tnsr \ExtAlg^aV) \to V \tnsr (\bbS_{\lambda} \tnsr \ExtAlg^kV)
\end{eqnarray*}

In other words this labeling of $\eta$ tells us to first separate off the box labeled "$V$" by applying appropriate component of comultiplication on exterior algebra. Then we separate off all the "$\ExtAlg$"'s. Precomposing with the inclusion $\bbS_{\eta} \into \ExtAlg^{\tilde{\eta}} V$ and postcomposing with the projection gives us the desired embedding.

Now let us think about this tensor product with brackets placed in a different way: $(V \tnsr \bbS_{\lambda}) \tnsr \ExtAlg^k V$. With this placement of the brackets we will think about each irreducible summand $\eta'$ as being obtained from the diagram of $\lambda$ by first adding one box (i.e. first tensor with $V$) and then adding $k$ more boxes (tensor with $\ExtAlg^k V$). Again, we will mark the box that was added first with the symbol "$V$" and the $k$ boxes which were added later will be marked with the symbol "$\ExtAlg$". 

\begin{example}
Now we think about the tensor product $V \tnsr \bbS_{(2,1,0)} \tnsr \ExtAlg^2 V$ with brackets placed as $(V \tnsr \bbS_{(2,1,0)}) \tnsr \ExtAlg^2 V$. The summands of $V \tnsr \bbS_{(2,1,0)}$ are:

{\scriptsize\young( \hfil \hfil V,\hfil)}, {\scriptsize\young( \hfil \hfil,\hfil V)}, {\scriptsize\young( \hfil \hfil,\hfil,V)}

Now we add two "$\ExtAlg$"'s according to the Pieri rule:

{\scriptsize\young( \hfil \hfil V\ExtAlg,\hfil \ExtAlg)}, {\scriptsize\young( \hfil \hfil V\ExtAlg,\hfil,\ExtAlg)}, {\scriptsize\young( \hfil \hfil V,\hfil \ExtAlg,\ExtAlg)}, \\

{\scriptsize\young( \hfil \hfil \ExtAlg,\hfil V\ExtAlg)}, {\scriptsize\young( \hfil \hfil \ExtAlg,\hfil V,\ExtAlg)},\\

{\scriptsize\young( \hfil \hfil \ExtAlg,\hfil \ExtAlg,V)}, {\scriptsize\young( \hfil \hfil \ExtAlg,\hfil,V,\ExtAlg)}, {\scriptsize\young( \hfil \hfil,\hfil \ExtAlg,V\ExtAlg)}, {\scriptsize\young( \hfil \hfil,\hfil \ExtAlg,V,\ExtAlg)}

\end{example}

Note that this time "$V$" is "inside" of "$\ExtAlg$"'s, i.e. all the "$\ExtAlg$" occur to the right and below the "$V$". Each such (labeled) diagram $\eta'$ defines an embedding $\bbS_{\eta'} \into (V \tnsr \bbS_{\lambda}) \tnsr \ExtAlg^k V$ in the following way: suppose $\eta$ has $a'_j$ $\ExtAlg$'s in the column $j$ and the $V$ is in the column $s'$. Let $a':=(a'_1, \cdots a'_t)$ and $b'$  be the vector all of whose components are zero, except $1$ in the position $s'$.

\begin{eqnarray*}
\bbS_{\eta'} \into \ExtAlg^{\tilde{\eta'}} V \xrightarrow{\Phi(\tilde{\eta'}, a')} \ExtAlg^{\tilde{\eta'}-a'}V \tnsr \ExtAlg^{a'}V \xrightarrow{\Phi(\tilde{\eta'}-a',b') \tnsr 1_{\ExtAlg^{a'}V}} (V \tnsr \ExtAlg^{\tilde{\eta'}-a'-b'}V)\tnsr \ExtAlg^{a'}V \to V \tnsr (\bbS_{\lambda} \tnsr \ExtAlg^kV)
\end{eqnarray*}

We will need to know "how to move "$V$" from outside to inside". More precisely, suppose we have a labeled diagram $\eta$ which comes from the bracket placement $V \tnsr (\bbS_{\lambda} \tnsr \ExtAlg^k V)$. This labeled diagram defines the embedding $\bbS_{\eta} \into V \tnsr (\bbS_{\lambda} \tnsr \ExtAlg^k V)$. The question is: which labbeled diagram $\eta'$ which comes from bracket placement $(V \tnsr \bbS_{\lambda}) \tnsr \ExtAlg^k V$ defines the same embedding as $\eta$? We will write $\eta \sim \eta'$ when $\eta$ and $\eta'$ define the same embedding into $V \tnsr \bbS_{\lambda} \tnsr \ExtAlg^k V$. 

The skew shape $\eta-\lambda$ has at most one row with two boxes. Suppose $\eta-\lambda$ has a row with two boxes, then multiplicity of such $\eta$ in $V \tnsr \bbS_{\lambda} \tnsr \ExtAlg^k V$ is equal to $1$, so there is only one labeled diagram $\nu$ coming from bracket placement $(V \tnsr \bbS_{\lambda}) \tnsr \ExtAlg^k V$ that has the same shape as $\eta$.

\begin{example}
In the above example we have the following equivalences:

{\scriptsize\young( \hfil \hfil \ExtAlg V,\hfil \ExtAlg)} $\sim$ {\scriptsize\young( \hfil \hfil V\ExtAlg,\hfil \ExtAlg)}, {\scriptsize\young( \hfil \hfil \ExtAlg,\hfil \ExtAlg V)} $\sim$ {\scriptsize\young( \hfil \hfil \ExtAlg,\hfil V\ExtAlg)}, {\scriptsize\young( \hfil \hfil \ExtAlg V,\hfil,\ExtAlg)} $\sim$ {\scriptsize\young( \hfil \hfil V\ExtAlg,\hfil,\ExtAlg)},{\scriptsize\young( \hfil \hfil,\hfil \ExtAlg,\ExtAlg V)} $\sim$ {\scriptsize\young( \hfil \hfil,\hfil \ExtAlg,V\ExtAlg)}. 

\end{example}

Suppose now that the skew-shape $\eta-\lambda$ has no row with two boxes. If there is no column that has both "$\ExtAlg$"'s and "$V$" then this $\eta$ appears in both bracket placements and no moving is necessary.

\begin{example}
In our example we see that the labeled diagrams

{\scriptsize\young( \hfil \hfil \ExtAlg,\hfil \ExtAlg,V)}, {\scriptsize\young( \hfil \hfil \ExtAlg,\hfil V,\ExtAlg)}, {\scriptsize\young( \hfil \hfil \ExtAlg,\hfil \ExtAlg,V)}.

appear in both bracket placements.
\end{example}

Finally, suppose that the skew shape $\eta-\lambda$ has no row with two boxes and some column has both "$\ExtAlg$"'s and "$V$". Then we need to "slide the "$V$" up". This does not change the embedding. One can easily see it by writing down the corresponding embeddings and using coassotiativity of $\ExtAlg V$ (see Figure ~\ref{fig:coass}).

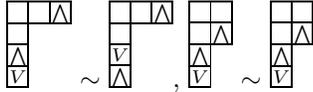
\begin{example}
{\scriptsize\young( \hfil \hfil \ExtAlg,\hfil,\ExtAlg,V)} $\sim$ {\scriptsize\young( \hfil \hfil \ExtAlg,\hfil,V,\ExtAlg)}, {\scriptsize\young( \hfil \hfil,\hfil \ExtAlg,\ExtAlg,V)} $\sim$ {\scriptsize\young( \hfil \hfil,\hfil \ExtAlg,\ExtAlg,V)}.
\end{example}

We will also need the Lemma 2.1 from \cite{SamWey} :

\begin{lemma}
\label{lem:Sam}
Let $\nu$ be a partition and take $\mu \in HS(\nu)$, i.e. the skew shape $\mu-\nu$ has no two boxes in the same comlumn. This means that there exist (unique) Pieri inclusion $\bbS_{\mu} \into R \tnsr \bbS_{\nu}$, which induces the map of $R$-modules $R \tnsr \bbS_{\mu} \to R \tnsr \bbS_{\nu}$. Suppose now we take $\eta$ to be any partition such that $\bbS_{\eta}$ occurs in both $R \tnsr \bbS_{\mu}$ and $R \tnsr \bbS_{\nu}$. Then the composition $\bbS_{\eta} \into R \tnsr \bbS_{\mu} \to R \tnsr \bbS_{\nu}$ is not zero. 
\end{lemma}

\section{Category of Equivariant Modules Over $\GL(V)$}
\label{sec:CatMod}

Set $G:=\GL(V)$, $R:=\Sym V$. Then $G$ acts on $R$ by ring automorphisms in the usual way. Let $A:=R \rtimes G$ denote the twisted group algebra. The multiplication is given by the rule $(r_1, g_1)(r_2, g_2)=(r_1g_1(r_2), g_1g_2)$. The grading on $A$ comes from the usual grading of $R$.

Let $N \in A-\fmod$ be a graded left $A$-module all whose components are finite dimensional vector spaces, then it is also a graded $R$-module via the inclusion $R \into A$ and a representation of $G$, moreover the multiplication map $\mu \colon R \tnsr N \to N$ is a map of representations. Conversely, a representation $N$ of $G$ with graded $R-$module structure such that the multiplication map $R \tnsr N \to M$ is a map of representations defines a graded module over the twisted group algebra.

\begin{definition}
We will call such a finitely generated module over $A$ an equivariant module.
\end{definition}

We will display the structure of  an equivariant module $N$ by drawing a lattice. The vertices of the lattice are the representations that occur in $N$. An arrow going from a representation $\bbS_{\eta}$ to a representation $\bbS_{\nu}$ means that $\bbS_{\nu} $ lies in the image of the map $V \tnsr \bbS_{\eta} \into V \tnsr N \to N$.

\begin{example}
Consider the equivariant module $R \tnsr \ExtAlg^3 V$. It's lattice is the following:

\begin{figure}[h!]

\begin{tikzpicture}[scale=1.4]

\node at (0,0) {{\scriptsize\young(\hfil,\hfil,\hfil)}};
\node at (0,1) {{\scriptsize\young(\hfil \hfil,\hfil,\hfil)}};
\node at (0,2) {{\scriptsize\young(\hfil \hfil \hfil,\hfil,\hfil)}};
\node at (0,3) {{\scriptsize\young(\hfil \hfil \hfil \hfil,\hfil,\hfil)}};
\node at (0,4) {{\scriptsize\young(\hfil \hfil \hfil \hfil \hfil,\hfil,\hfil)}};
\node at (0,5) {$\vdots$};

\node at (2,1) {{\scriptsize\young(\hfil,\hfil,\hfil,\hfil)}};
\node at (2,2) {{\scriptsize\young(\hfil \hfil,\hfil,\hfil,\hfil)}};
\node at (2,3) {{\scriptsize\young(\hfil \hfil \hfil,\hfil,\hfil,\hfil)}};
\node at (2,4) {{\scriptsize\young(\hfil \hfil \hfil \hfil,\hfil,\hfil,\hfil)}};
\node at (2,5) {$\vdots$};

\path[->,font=\scriptsize] 

(0,0) edge[shorten <= .5cm, shorten >= .5cm] (0,1)
(0,1) edge[shorten <= .5cm, shorten >= .5cm] (0,2) 
(0,2) edge[shorten <= .5cm, shorten >= .5cm] (0,3) 
(0,3) edge[shorten <= .5cm, shorten >= .5cm] (0,4) 
(0,4) edge[shorten <= .5cm, shorten >= .5cm] (0,5)

(0,0) edge[shorten <= .5cm, shorten >= .5cm] (2,1) 
(0,1) edge[shorten <= .5cm, shorten >= .5cm] (2,2) 
(0,2) edge[shorten <= .5cm, shorten >= .5cm] (2,3) 
(0,3) edge[shorten <= .5cm, shorten >= .5cm] (2,4) 
(0,4) edge[shorten <= .5cm, shorten >= .5cm] (2,5)

(2,1) edge[shorten <= .5cm, shorten >= .5cm] (2,2)
(2,2) edge[shorten <= .5cm, shorten >= .5cm] (2,3) 
(2,3) edge[shorten <= .5cm, shorten >= .5cm] (2,4) 
(2,4) edge[shorten <= .5cm, shorten >= .5cm] (2,5);

\end{tikzpicture}

\caption{$R \tnsr \ExtAlg^3V$}

\end{figure}
\FloatBarrier

We see that whenever there is a posibility of an arrow we do have an arrow. This follows from the Lemma ~\ref{lem:Sam}.

\end{example}

\begin{example}

In fact all the modules of the form $R \tnsr \bbS_{\lambda}$ have this kind of lattice (see next section for precise statement). Here is another example:

\begin{figure}[h!]

\begin{tikzpicture}[scale=1.4]

\node at (0,0) {{\scriptsize\young(\hfil \hfil)}};
\node at (0,1) {{\scriptsize\young(\hfil \hfil \hfil)}};
\node at (0,2) {{\scriptsize\young(\hfil \hfil \hfil \hfil)}};
\node at (0,3) {{\scriptsize\young(\hfil \hfil \hfil \hfil \hfil)}};
\node at (0,4) {{\scriptsize\young(\hfil \hfil \hfil \hfil \hfil \hfil)}};
\node at (0,5) {$\vdots$};

\node at (2,1) {{\scriptsize\young(\hfil \hfil,\hfil)}};
\node at (2,2) {{\scriptsize\young(\hfil \hfil \hfil,\hfil)}};
\node at (2,3) {{\scriptsize\young(\hfil \hfil \hfil \hfil,\hfil)}};
\node at (2,4) {{\scriptsize\young(\hfil \hfil \hfil \hfil \hfil,\hfil)}};
\node at (2,5) {$\vdots$};

\node at (4,2) {{\scriptsize\young(\hfil \hfil,\hfil \hfil)}};
\node at (4,3) {{\scriptsize\young(\hfil \hfil \hfil,\hfil \hfil)}};
\node at (4,4) {{\scriptsize\young(\hfil \hfil \hfil \hfil,\hfil \hfil)}};
\node at (4,5) {$\vdots$};

\path[->,font=\scriptsize] 

(0,0) edge[shorten <= .5cm, shorten >= .5cm] (0,1)
(0,1) edge[shorten <= .5cm, shorten >= .5cm] (0,2) 
(0,2) edge[shorten <= .5cm, shorten >= .5cm] (0,3) 
(0,3) edge[shorten <= .5cm, shorten >= .5cm] (0,4) 
(0,4) edge[shorten <= .5cm, shorten >= .5cm] (0,5)

(0,0) edge[shorten <= .5cm, shorten >= .5cm] (2,1)
(0,1) edge[shorten <= .5cm, shorten >= .5cm] (2,2) 
(0,2) edge[shorten <= .5cm, shorten >= .5cm] (2,3) 
(0,3) edge[shorten <= .5cm, shorten >= .5cm] (2,4) 
(0,4) edge[shorten <= .5cm, shorten >= .5cm] (2,5)

(2,1) edge[shorten <= .5cm, shorten >= .5cm] (2,2)
(2,2) edge[shorten <= .5cm, shorten >= .5cm] (2,3) 
(2,3) edge[shorten <= .5cm, shorten >= .5cm] (2,4) 
(2,4) edge[shorten <= .5cm, shorten >= .5cm] (2,5)

(2,1) edge[shorten <= .5cm, shorten >= .5cm] (4,2)
(2,2) edge[shorten <= .5cm, shorten >= .5cm] (4,3) 
(2,3) edge[shorten <= .5cm, shorten >= .5cm] (4,4) 
(2,4) edge[shorten <= .5cm, shorten >= .5cm] (4,5) 

(4,2) edge[shorten <= .5cm, shorten >= .5cm] (4,3)
(4,3) edge[shorten <= .5cm, shorten >= .5cm] (4,4) 
(4,4) edge[shorten <= .5cm, shorten >= .5cm] (4,5); 

\end{tikzpicture}

\caption{$R \tnsr \Sym_2 V$}

\end{figure}
\FloatBarrier

\end{example}

\section{Projectives}
\label{sec:Proj}

We are working in the category of graded equivariant modules. In each degree such a module $N$ is just a finite-dimensional representation of $G$, so in each degree it decomposes into irreducible representations (see Section ~\ref{subsec:Rep}). Let $(\_)^G \niso \Hom_G(\Bbbk, \_) \colon A-\fmod \to \Bbbk-\fmod$ denote the functor which takes each $A$-module to the (graded!) vector space of its $G$-invariants. The functor $N \mapsto N^G$ picks out the copies of trivial representation $\Bbbk$ from $N$ in each degree, thus it is exact.

The morphisms in $A-\fmod$ are also easy to describe:

\[
\Hom_A(M,N) \niso \Hom_R(M,N)^G
\]

Also, $P \in A-\fmod$ is projective if and only if $P$ is projective as an $R$-module, i.e. is free $R$-module. Let $\lambda$ be a partition and $\bbS_{\lambda}$ the corresponding irreducible representation of $G$. Then the $A$-module $P_{\lambda}:=R \underset{\Bbbk}{\tnsr} \bbS_{\lambda}$ is projective and indecomposable because its ring of $A$-endomorphisms is just $\Bbbk$. Moreover, for any $A$-module $N$ there is a surjection $\underset{\lambda \in I}{\Drsum} P_{\lambda} \onto N$ where $\lambda$ varies over some finite set $I$. For example, this set can be taken to be the set of all representations in $N$ which contain generators.

The structure of $P_{\lambda}$ is quite easy to understand, thanks to Lemma ~\ref{lem:Sam}. Let us pick some representation $\bbS_{\eta} \into P_{\lambda}$ which occurs in $P_{\lambda}$. We would like to know what is the next-to-lowest degree of the submodule generated by $\bbS_{\eta}$ in $P_{\lambda}$. In other words we would like to know what is the image of the composition $V \tnsr \bbS_{\eta} \into V \tnsr P_{\lambda} \to P_{\lambda}$ where the first map is the given inclusion $\bbS_{\eta} \into P_{\lambda}$ tensored with $V$ and the second map is the multiplication on $P_{\lambda}$. The answer is that it consists of all the representations in $P_{\lambda}$ which can be obtained from $\bbS_{\eta}	$ by the Pieri rule.

\section{Elementary Equivariant Module Associated with a Partition $\lambda$.}
\label{sec:Elem}

Let $M=M_{\lambda}$ be the (graded) $R$- module defined in the following way. As a vector space $M= \underset{\mu \in R_{\lambda}} {\Drsum} \bbS_{\mu}$ where $R_{\lambda}:= \{\mu | \text{ $\mu$ is obtained from $\lambda$ by adding boxes to the first row}\}$. The multiplication map just adds another box to the first row of each $\mu$. Let us describe the multiplication map more rigorously: first we want to define the map of vector spaces $V \underset{\Bbbk}{\tnsr} M_{\lambda} \to M_{\lambda}$. When restricted to the summand $V \tnsr \bbS_{\mu}$ of $V \tnsr M_{\lambda}$ we define it to be the projection $V \tnsr \bbS_{\mu} \to \bbS_{\mu+(1,0, \cdots)}$ where $\mu+(1,0, \cdots)$ is the partition obtained from $\mu$ by adding one box to the first row. Now it is clear that applying the map twice $V \tnsr (V \tnsr M_{\lambda}) \to V \tnsr M_{\lambda} \to M_{\lambda}$ factors through $\Sym_2 V \tnsr M_{\lambda} \to M_{\lambda}$, so the multiplication map so defined indeed gives $M_{\lambda}$ the structure of an equivariant module over $R$. 

There is also a less direct way to see that this indeed defines an $R$-module structure. In Section~\ref{sec:Filt} we will see that each $M_{\lambda}$ can be realized as a submodule of $P_{\underline{\lambda}}$ where $\underline{\lambda}$ is obtained from $\lambda$ be removing a box from each non-empty column. Also see Subsection~\ref{subsec:Geom}.

\begin{definition}
We will call the module $M_{\lambda}$ defined above the elementary equivariant module associated with partition $\lambda$.
\end{definition}

It follows from the definition that $M_{\lambda}$ is generated in its lowest degree and in each degree it has one representation. 

\begin{example}
Consider the module of K\"{a}hler differentials. It is often defined as the kernel of the canonical map $R \tnsr V \to R$. The lattice of $R \tnsr V$ is shown below:

\begin{figure}[h!]

\begin{tikzpicture}[scale=1.4]

\node at (0,0) {{\scriptsize\young(\hfil)}};
\node at (0,1) {{\scriptsize\young(\hfil \hfil)}};
\node at (0,2) {{\scriptsize\young(\hfil \hfil \hfil)}};
\node at (0,3) {{\scriptsize\young(\hfil \hfil \hfil \hfil)}};
\node at (0,4) {{\scriptsize\young(\hfil \hfil \hfil \hfil \hfil)}};
\node at (0,5) {$\vdots$};

\node at (2,1) {{\scriptsize\young(\hfil,\hfil)}};
\node at (2,2) {{\scriptsize\young(\hfil \hfil,\hfil)}};
\node at (2,3) {{\scriptsize\young(\hfil \hfil \hfil,\hfil)}};
\node at (2,4) {{\scriptsize\young(\hfil \hfil \hfil \hfil,\hfil)}};
\node at (2,5) {$\vdots$};

\path[->,font=\scriptsize] 

(0,0) edge[shorten <= .5cm, shorten >= .5cm] (0,1)
(0,1) edge[shorten <= .5cm, shorten >= .5cm] (0,2) 
(0,2) edge[shorten <= .5cm, shorten >= .5cm] (0,3) 
(0,3) edge[shorten <= .5cm, shorten >= .5cm] (0,4) 
(0,4) edge[shorten <= .5cm, shorten >= .5cm] (0,5)

(0,0) edge[shorten <= .5cm, shorten >= .5cm] (2,1) 
(0,1) edge[shorten <= .5cm, shorten >= .5cm] (2,2) 
(0,2) edge[shorten <= .5cm, shorten >= .5cm] (2,3) 
(0,3) edge[shorten <= .5cm, shorten >= .5cm] (2,4) 
(0,4) edge[shorten <= .5cm, shorten >= .5cm] (2,5)

(2,1) edge[shorten <= .5cm, shorten >= .5cm] (2,2)
(2,2) edge[shorten <= .5cm, shorten >= .5cm] (2,3) 
(2,3) edge[shorten <= .5cm, shorten >= .5cm] (2,4) 
(2,4) edge[shorten <= .5cm, shorten >= .5cm] (2,5);

\end{tikzpicture}

\caption{$R \tnsr V$}

\end{figure}
\FloatBarrier

Clearly, the kernel of the canonical map $R \tnsr V \to R$ consists of all the representations in the right column, thus the module of K\"{a}hler differentials is just $M_{(1,1,0, \cdots)}$.
\end{example}

\begin{definition}

From the lattice of $M_{\lambda}$ it is clear that the only equivariant submodules of $M_{\lambda}$ are isomorphic to $M_{\lambda+(l,0, \cdots)}$. We will denote the  $l$-th power of the maximal ideal generated by  $V$ in $R$ by $V^l$ and so the submodule generated by $\bbS_{\lambda+(l,0,\cdots)}$ in $M_{\lambda}$ will be denoted by $V^lM_{\lambda}$. This also means that the only equivariant quotients of $M_{\lambda}$ are $M_{\lambda}/V^lM_{\lambda}$. We will call such module a $l$-truncation of $M_{\lambda}$.

\end{definition}

The importance of the modules $M_{\lambda}$ and their truncations is explained in Theorem~\ref{thm:filt} below.

Now we consider an interesting phenomenon to which we will refer to as "splicing". Let us illustrate it with an example:

\begin{example}
\label{ex:spl}

Take two elementary modules, say $M_{(2,1)}$ and $M_{(3,1)}$. Note that $M_{(3,1)}$ can be imbedded as a submodule of $M_{(2,1)}$. Let us draw the lattice of the direct sum $M_{(3,1)} \drsum M_{(2,1)}$:

\begin{figure}[h!]

\begin{tikzpicture}[scale=1.4]

\node at (0,1) {{\scriptsize\young(\hfil \hfil \hfil,\hfil)}};
\node at (0,2) {{\scriptsize\young(\hfil \hfil \hfil \hfil,\hfil)}};
\node at (0,3) {{\scriptsize\young(\hfil \hfil \hfil \hfil \hfil,\hfil)}};
\node at (0,4) {{\scriptsize\young(\hfil \hfil \hfil \hfil \hfil \hfil,\hfil)}};
\node at (0,5) {$\vdots$};

\node at (-2,0) {{\scriptsize\young(\hfil \hfil,\hfil)}};
\node at (-2,1) {{\scriptsize\young(\hfil \hfil \hfil,\hfil)}};
\node at (-2,2) {{\scriptsize\young(\hfil \hfil \hfil \hfil,\hfil)}};
\node at (-2,3) {{\scriptsize\young(\hfil \hfil \hfil \hfil \hfil,\hfil)}};
\node at (-2,4) {{\scriptsize\young(\hfil \hfil \hfil \hfil \hfil \hfil,\hfil)}};
\node at (-2,5) {$\vdots$};

\path[->,font=\scriptsize] 

(0,1) edge[shorten <= .5cm, shorten >= .5cm] (0,2)
(0,2) edge[shorten <= .5cm, shorten >= .5cm] (0,3) 
(0,3) edge[shorten <= .5cm, shorten >= .5cm] (0,4) 
(0,4) edge[shorten <= .5cm, shorten >= .5cm] (0,5) 

(-2,0) edge[shorten <= .5cm, shorten >= .5cm] (-2,1) 
(-2,1) edge[shorten <= .5cm, shorten >= .5cm] (-2,2) 
(-2,2) edge[shorten <= .5cm, shorten >= .5cm] (-2,3) 
(-2,3) edge[shorten <= .5cm, shorten >= .5cm] (-2,4) 
(-2,4) edge[shorten <= .5cm, shorten >= .5cm] (-2,5);

\end{tikzpicture}

\caption{$M_{(3,1)} \drsum M_{(2,1)}$}

\end{figure}
\FloatBarrier

Now let us embed $\bbS_{(5,1)}$ into $M_{(3,1)} \drsum M_{(2,1)}$ by diagonal embedding: we define the map $\iota \colon \bbS_{(5,1)} \to \bbS_{(5,1)} \drsum \bbS_{(5,1)}$ to be identity on both components. Now the submodule generated by the image $\iota$ is clearly isomorphic to $M_{(5,1)}$, but in each degree this submodule embeds into the direct sum diagonally. Now let us take the quotient. We get a an interesting equivariant module! (it's lattice is on the right) 

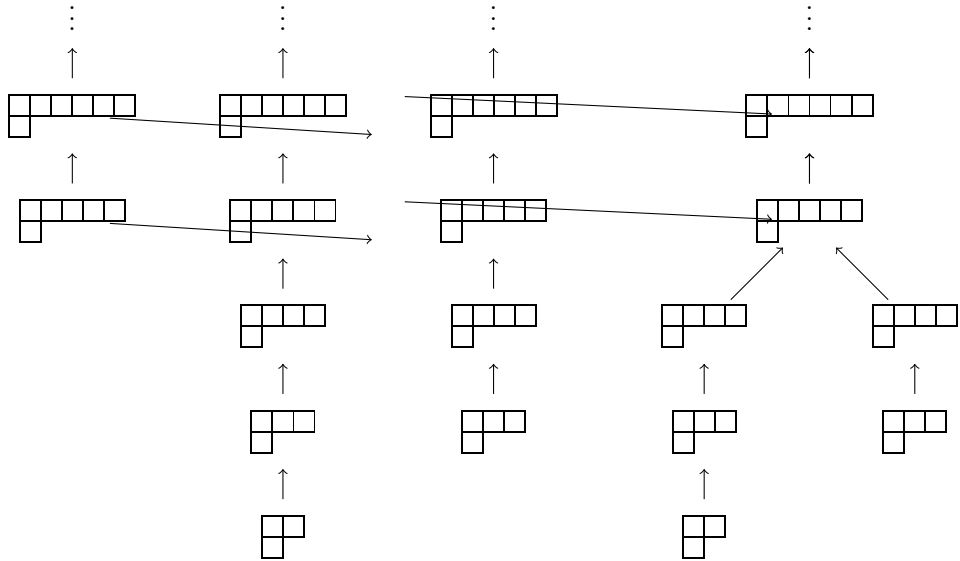
\begin{figure}[h!]

\begin{tikzpicture}[scale=1.4]

\node at (0,1) {{\scriptsize\young(\hfil \hfil \hfil,\hfil)}};
\node at (0,2) {{\scriptsize\young(\hfil \hfil \hfil \hfil,\hfil)}};
\node at (-1,3) {{\scriptsize\young(\hfil \hfil \hfil \hfil \hfil,\hfil)}};
\node at (-1,4) {{\scriptsize\young(\hfil \hfil \hfil \hfil \hfil \hfil,\hfil)}};
\node at (-1,5) {$\vdots$};

\node at (-2,0) {{\scriptsize\young(\hfil \hfil,\hfil)}};
\node at (-2,1) {{\scriptsize\young(\hfil \hfil \hfil,\hfil)}};
\node at (-2,2) {{\scriptsize\young(\hfil \hfil \hfil \hfil,\hfil)}};

\node at (-4,1) {{\scriptsize\young(\hfil \hfil \hfil,\hfil)}};
\node at (-4,2) {{\scriptsize\young(\hfil \hfil \hfil \hfil,\hfil)}};
\node at (-4,3) {{\scriptsize\young(\hfil \hfil \hfil \hfil \hfil,\hfil)}};
\node at (-4,4) {{\scriptsize\young(\hfil \hfil \hfil \hfil \hfil \hfil,\hfil)}};
\node at (-4,5) {$\vdots$};

\node at (-6,0) {{\scriptsize\young(\hfil \hfil,\hfil)}};
\node at (-6,1) {{\scriptsize\young(\hfil \hfil \hfil,\hfil)}};
\node at (-6,2) {{\scriptsize\young(\hfil \hfil \hfil \hfil,\hfil)}};
\node at (-6,3) {{\scriptsize\young(\hfil \hfil \hfil \hfil \hfil,\hfil)}};
\node at (-6,4) {{\scriptsize\young(\hfil \hfil \hfil \hfil \hfil \hfil,\hfil)}};
\node at (-6,5) {$\vdots$};

\node at (-8,3) {{\scriptsize\young(\hfil \hfil \hfil \hfil \hfil,\hfil)}};
\node at (-8,4) {{\scriptsize\young(\hfil \hfil \hfil \hfil \hfil \hfil,\hfil)}};
\node at (-8,5) {$\vdots$};

\path[->,font=\scriptsize] 

(0,1) edge[shorten <= .5cm, shorten >= .5cm] (0,2)
(0,2) edge[shorten <= .5cm, shorten >= .5cm] (-1,3) 
(-1,3) edge[shorten <= .5cm, shorten >= .5cm] (-1,4)
(-1,4) edge[shorten <= .5cm, shorten >= .5cm] (-1,5)

(-2,0) edge[shorten <= .5cm, shorten >= .5cm] (-2,1) 
(-2,1) edge[shorten <= .5cm, shorten >= .5cm] (-2,2) 
(-2,2) edge[shorten <= .5cm, shorten >= .5cm] (-1,3) 
(-1,3) edge[shorten <= .5cm, shorten >= .5cm] (-1,4) 
(-1,4) edge[shorten <= .5cm, shorten >= .5cm] (-1,5)

(-4,1) edge[shorten <= .5cm, shorten >= .5cm] (-4,2)
(-4,2) edge[shorten <= .5cm, shorten >= .5cm] (-4,3) 
(-4,3) edge[shorten <= .5cm, shorten >= .5cm] (-4,4) 
(-4,4) edge[shorten <= .5cm, shorten >= .5cm] (-4,5) 

(-6,0) edge[shorten <= .5cm, shorten >= .5cm] (-6,1) 
(-6,1) edge[shorten <= .5cm, shorten >= .5cm] (-6,2) 
(-6,2) edge[shorten <= .5cm, shorten >= .5cm] (-6,3) 
(-6,3) edge[shorten <= .5cm, shorten >= .5cm] (-6,4) 
(-6,4) edge[shorten <= .5cm, shorten >= .5cm] (-6,5)

(-8,3) edge[shorten <= .5cm, shorten >= .5cm] (-8,4)
(-8,4) edge[shorten <= .5cm, shorten >= .5cm] (-8,5)

(-8,3) edge[shorten <= .5cm, shorten >= .5cm](-4.8,2.8)
(-5.2,3.2) edge[shorten <= .5cm, shorten >= .5cm] (-1,3)

(-8,4) edge[shorten <= .5cm, shorten >= .5cm](-4.8,3.8)
(-5.2,4.2) edge[shorten <= .5cm, shorten >= .5cm] (-1,4);

\end{tikzpicture}

\caption{The module on the right is obtained by splicing $M_{(3,1)}$ with $M_{(2,1)}$ in degree $6$}

\end{figure}
\FloatBarrier
We will say that the module on the right is obtained by splicing $M_{(3,1)}$ with $M_{(2,1)}$ in degree $6$. One can now see how more than two $M_{\lambda}$'s could be spliced. 
\end{example}

\begin{example}
Suppose now we have a submodule $K \into \underset{\lambda \in J}{\Drsum}M_{\lambda}$ and we want to take the quotient. $K$ is generated by a finite number of representations $\bbS_{\lambda_1}, \bbS_{\lambda_2}, \cdots \bbS_{\lambda_m}$ each of which occurs in $\underset{\lambda \in J}{\Drsum}M_{\lambda}$. Let us order $\bbS_{\lambda_1}, \bbS_{\lambda_2}, \cdots \bbS_{\lambda_m}$ by the degree (which is the number of boxes), say $\bbS_{\lambda_1} \cdots \bbS_{\lambda_p}$ are all the generators of $K$ which have lowest degree. We will take the quotient by first dividing out $<\bbS_{\lambda_1}>$ - the submodule generated by $\bbS_{\lambda_1}$ in $\underset{\lambda \in J_k}{\Drsum}M_{\lambda}$, then dividing out the submodule generated by $\bbS_{\lambda_2}$ in $\underset{\lambda \in J_k}{\Drsum}M_{\lambda}/<\bbS_{\lambda_1}>$ and so on.

The lattice of $\underset{\lambda \in J_k}{\Drsum}M_{\lambda}$ looks like several vertical branches, each starting in some degree. Looking back at Example~\ref{ex:spl} we see that dividing out by $<\bbS_{\lambda_1}>$ has the effect that some branches are spliced together at some degree and possibly truncated to get the lattice of $\underset{\lambda \in J_k}{\Drsum}M_{\lambda}/<\bbS_{\lambda_1}>$. Now dividing out by $<\bbS_{\lambda_2}>$ just splices or truncates more branches at the same degree, so once we are done with all of $\bbS_{\lambda_1} \cdots \bbS_{\lambda_p}$ the lattice has some branches spliced or truncated at the same degree. Similar things will happen when we go on with generators in the higher degrees.

\end{example}

\section{Filtration of an Equivariant Module}
\label{sec:Filt}

We will need the following definitions.

\begin{definition}
Let $\lambda$ be a partition of $d$ and let $t$ be the number of columns of $\lambda$. The saturation of $\lambda$ denoted by $\overline{\lambda}$ is the partition of $d+t$ obtained from $\lambda$ by putting one box in each column.

For $0 \leq i \leq t$ and define $S(\lambda, i)$ to be the set of partitions $\beta$ (of $d+t-i$) such that:

\begin{enumerate}
\item The skew-shape $\beta-\lambda$ has at most one box in each column, and

\item The skew shape $\beta-\lambda$ needs $i$ boxes to be added to it to get the skew shape $\overline{\lambda}-\lambda$. Note that $S(\lambda, t)=\{\lambda\}$ and $S(\lambda, 0)=\{\overline{\lambda}\}$
\end{enumerate}

Define $S(\lambda, i)$ to be the empty set if $i$ is not in the range $0 \leq i \leq t$.

\end{definition}

\begin{example}
Let $\lambda={\scriptsize\young(\hfil \hfil,\hfil)}$ Then $\overline{\lambda}={\scriptsize\young(\hfil \hfil,\hfil \hfil,\hfil)}$, so $S(\lambda,0)=\{ {\scriptsize\young(\hfil \hfil,\hfil \hfil,\hfil)} \}$, $S(\lambda,1)=\{ {\scriptsize\young(\hfil \hfil,\hfil \hfil)}, {\scriptsize\young(\hfil \hfil,\hfil,\hfil)} \}$ and $S(\lambda, 2)=\{ {\scriptsize\young(\hfil \hfil,\hfil)} \}$.	
\end{example}

Now consider the module $P_{\lambda}=R \underset{\Bbbk}{\tnsr} \bbS_{\lambda}$. The submodule $<\bbS_{\overline{\lambda}}> $ generated by $\bbS_{\overline{\lambda}}$ in $P_{\lambda}$ is isomorphic to $M_{\overline{\lambda}}$ because by Pieri rule the only possibility to add new boxes to $\overline{\lambda}$ is to put them all in the first row, thus creating new columns. 

\begin{example}
Suppose dimension of $V$ is $3$ and consider $P_{(2,0,0)}:=R \tnsr \Sym_2 V$. The saturation of $\lambda=(2,0,0)$ is the partition $(2,2,0)$. The representation $\bbS_{(2,2,0)}$ has dimemsion $6$ with the basis consisting of

\[
{\scriptsize\young(11,22)}, {\scriptsize\young(11,23)}, {\scriptsize\young(11,33)}, {\scriptsize\young(12,23)}, {\scriptsize\young(12,33)}, {\scriptsize\young(22,33)}
\]

The images of these elements in $R \tnsr \Sym_2 V$ are:

\begin{flalign*}
& {\scriptsize\young(11,22)} \mapsto x_1^2 \tnsr x_2^2-x_1x_2 \tnsr x_2x_1-x_2x_1 \tnsr x_1x_2+x_2^2 \tnsr x_1^2  \\
&{\scriptsize\young(11,23)} \mapsto x_1^2 \tnsr x_2x_3-x_1x_3 \tnsr x_2x_1-x_2x_1 \tnsr x_1x_3+x_2x_3 \tnsr x_1^2 \\
&{\scriptsize\young(11,33)} \mapsto x_1^2 \tnsr x_3^2-x_1x_3 \tnsr x_3x_1-x_3x_1 \tnsr x_1x_3+x_3^2 \tnsr x_1^2  \\
&{\scriptsize\young(12,23)} \mapsto x_1x_2 \tnsr x_2x_3-x_1x_3 \tnsr x_2^2-x_2^2 \tnsr x_1x_3+x_2x_3 \tnsr x_1x_2 \\
&{\scriptsize\young(12,33)} \mapsto x_1x_2 \tnsr x_3^2-x_1x_3 \tnsr x_3x_2-x_3x_2 \tnsr x_1x_3+x_3^2 \tnsr x_1x_2 \\
&{\scriptsize\young(22,33)} \mapsto x_2^2 \tnsr x_3^2-x_2x_3 \tnsr x_3x_2-x_3x_2 \tnsr x_2x_3+x_3^2 \tnsr x_2^2 
\end{flalign*}

The submodule generated by these elements in $R \tnsr \Sym_2 V$ is isomorphic to $M_{(2,2,0)}$.
\end{example}

Consider the quotient $P_{\lambda}/<\bbS_{\overline{\lambda}}>$ and take some $\beta$ from $S(\lambda, 1)$, say the skew shape $\beta-\lambda$ needs one box to be added to column $j$ to get $\overline{\lambda}-\lambda$. Let us look at the submodule generated by such $\beta$ in $P_{\lambda}/<\bbS_{\overline{\lambda}}>$. It consists of representations which occur in $P_{\lambda}$ obtained from $\beta$ by adding new boxes by the Pieri rule. If one of the boxes gets added to the column $j$ (this makes $\overline{\lambda}$ from $\beta$) all the other boxes have to go to the first row and the representation obtained in this way lies in $<\bbS_{\overline{\lambda}}>$, so it is zero in $P_{\lambda}/<\bbS_{\overline{\lambda}}>$. Thus the submodule generated by such $\beta$ in $P_{\lambda}/<\bbS_{\overline{\lambda}}>$ is isomorphic to $M_{\beta}$.

\begin{example}
Let us look at the quotient $(R \tnsr \Sym_2 V)/<\bbS_{(2,2,0)}>$. Its lattice looks like so:

\begin{figure}[h!]

\begin{tikzpicture}[scale=1.4]

\node at (0,0) {{\scriptsize\young(\hfil \hfil)}};
\node at (0,1) {{\scriptsize\young(\hfil \hfil \hfil)}};
\node at (0,2) {{\scriptsize\young(\hfil \hfil \hfil \hfil)}};
\node at (0,3) {{\scriptsize\young(\hfil \hfil \hfil \hfil \hfil)}};
\node at (0,4) {{\scriptsize\young(\hfil \hfil \hfil \hfil \hfil \hfil)}};
\node at (0,5) {$\vdots$};

\node at (2,1) {{\scriptsize\young(\hfil \hfil,\hfil)}};
\node at (2,2) {{\scriptsize\young(\hfil \hfil \hfil,\hfil)}};
\node at (2,3) {{\scriptsize\young(\hfil \hfil \hfil \hfil,\hfil)}};
\node at (2,4) {{\scriptsize\young(\hfil \hfil \hfil \hfil \hfil,\hfil)}};
\node at (2,5) {$\vdots$};

\path[->,font=\scriptsize] 

(0,0) edge[shorten <= .5cm, shorten >= .5cm] (0,1)
(0,1) edge[shorten <= .5cm, shorten >= .5cm] (0,2) 
(0,2) edge[shorten <= .5cm, shorten >= .5cm] (0,3) 
(0,3) edge[shorten <= .5cm, shorten >= .5cm] (0,4) 
(0,4) edge[shorten <= .5cm, shorten >= .5cm] (0,5)

(0,0) edge[shorten <= .5cm, shorten >= .5cm] (2,1)
(0,1) edge[shorten <= .5cm, shorten >= .5cm] (2,2) 
(0,2) edge[shorten <= .5cm, shorten >= .5cm] (2,3) 
(0,3) edge[shorten <= .5cm, shorten >= .5cm] (2,4) 
(0,4) edge[shorten <= .5cm, shorten >= .5cm] (2,5)

(2,1) edge[shorten <= .5cm, shorten >= .5cm] (2,2)
(2,2) edge[shorten <= .5cm, shorten >= .5cm] (2,3) 
(2,3) edge[shorten <= .5cm, shorten >= .5cm] (2,4) 
(2,4) edge[shorten <= .5cm, shorten >= .5cm] (2,5);

\end{tikzpicture}

\caption{$(R \tnsr \Sym_2 V)/<\bbS_{(2,2,0)}>$}

\end{figure}
\FloatBarrier

When dimension of $V$ is $3$ the represenation $\bbS_{(2,1,0)}$ has dimension $8$ and the standard basis is:

\[
{\scriptsize\young(11,2)}, {\scriptsize\young(11,3)}, {\scriptsize\young(12,2)}, {\scriptsize\young(12,3)}, {\scriptsize\young(13,2)}, {\scriptsize\young(13,3)}, {\scriptsize\young(22,3)}, {\scriptsize\young(23,3)}
\]

The images of these elements in $R \tnsr \Sym_2 V$ are:

\begin{flalign*}
&{\scriptsize\young(11,2)} \mapsto x_2 \tnsr x_1^2-x_1 \tnsr x_2x_1\\
&{\scriptsize\young(11,3)} \mapsto x_3 \tnsr x_1^2-x_1 \tnsr x_1x_3\\
&{\scriptsize\young(12,2)} \mapsto x_2 \tnsr x_1x_2-x_1 \tnsr x_2^2\\
&{\scriptsize\young(12,3)} \mapsto x_3 \tnsr x_1x_2-x_1 \tnsr x_2x_3 \\
&{\scriptsize\young(13,2)} \mapsto x_2 \tnsr x_1x_3-x_1 \tnsr x_2x_3\\
&{\scriptsize\young(13,3)} \mapsto x_3 \tnsr x_1x_3-x_1 \tnsr x_3^2\\
&{\scriptsize\young(22,3)} \mapsto x_3 \tnsr x_2x_3-x_2 \tnsr x_3^2\\
&{\scriptsize\young(23,3)} \mapsto x_3 \tnsr x_2x_3-x_2 \tnsr x_3^2\\
\end{flalign*}

Thus the images of these elements in the quotient $(R \tnsr \Sym_2 V)/<\bbS_{(2,2,0)}>$ generate a submodule isomorphic to $M_{(2,1,0)}$.
\end{example}

Now the proof of the following lemma should be clear:

\begin{lemma}
The module $P_{\lambda}$ has a filtration:

\[
F_{\lambda}^{\bullet} \colon 0 \into F_{\lambda}^{0} \into F_{\lambda}^{1} \into \cdots \into F_{\lambda}^{t-1} \into F_{\lambda}^{t}=P_{\lambda}
\]

with the associated graded $\underset{i \in \bbZ}{\Drsum}F_{\lambda}^{i}/F_{\lambda}^{i-1} \niso \underset{i \in \bbZ}{\Drsum} \underset{\beta \in S(\lambda, i)}{\drsum} M_{\beta}$

In fact, $F_{\lambda}^i = <\underset{\beta \in S(\lambda, i)}{\Drsum} \bbS_{\beta}>$ - the submodule generated by all representations $\bbS_{\beta}$ where $\beta \in S(\lambda, i)$.
\end{lemma}

Now let $N$ be any finitely generated equivariant module. Then there exists a surjection $\pi \colon \underset{\lambda \in I}{\Drsum} P_{\lambda} \onto N$  where $I$ is some finite set of partitions (possibly with some elements occuring more than once). Let $F_{\lambda}^{\bullet}$ to be the filtration of the summand $P_{\lambda}$ from the above lemma. Now define the filtration $F^{\bullet}$ of $\underset{\lambda \in I}{\Drsum} P_{\lambda}$ to be $F^i:=\underset{\lambda \in I}{\Drsum} F_{\lambda}^i$. Now push this filtration forward to $N$ to get a filtration $F_N^{\bullet}$: $F_N^i:= \im \pi|_{F_i}$. Let us illustrate this with a diagram:

\begin{figure}[h!]
\begin{tikzpicture}
\matrix(m) [matrix of math nodes, 
row sep=3.0em, column sep=3.0em, 
text height=1.5ex, text depth=0.25ex]
{0 & 0 &\\
F^i/F^{i-1} & F_N^i/F_N^{i-1} & \\
F^i & F_N^i & 0 \\
F^{i-1} & F_N^{i-1} & 0 \\
0 & 0 &\\};

\path[->,font=\scriptsize] 
(m-2-1) edge (m-2-2)
(m-3-1) edge node[above]{$\pi$}(m-3-2)
(m-3-2) edge (m-3-3)
(m-4-1) edge node[above]{$\pi$} (m-4-2) 
(m-4-2) edge (m-4-3)
(m-5-1) edge (m-4-1)
(m-4-1) edge (m-3-1)
(m-3-1) edge (m-2-1)
(m-2-1) edge (m-1-1)

(m-5-2) edge (m-4-2)
(m-4-2) edge (m-3-2)
(m-3-2) edge (m-2-2)
(m-2-2) edge (m-1-2);
\end{tikzpicture}
\end{figure}
\FloatBarrier

After chasing the diagram, we see that the induced map  $F^i/F^{i-1} \to F_N^i/F_N^{i-1}$ is a surjection. 
Now let us prove the following:

\begin{theorem}
\label{thm:filt}
Any finitely generated equivariant module $N$ has a filtration with the associated graded being the direct sum of modules of only two kinds: either $M_{\lambda}$ or truncations of $M_{\lambda}$.
\end{theorem}

\begin{proof}
Filter each $F^i/F^{i-1}$ so that the associated graded has just one $M_{\lambda}$ in each piece. Lift this to get a refinement of filtration $F^{\bullet}$ of $\underset{\lambda \in I}{\Drsum} P_{\lambda}$. Pushing this to $N$ gives the desired filtration. 

\end{proof}

\section{Resolution of $M_{\lambda}$}
\label{sec:Res}

Consider now the module $M_{\lambda} \tnsr \ExtAlg^k V$. It is generated in its lowest degree. Let us take the summand $\bbS_{\alpha}$ of $\bbS_{\mu} \tnsr \ExtAlg^k V$. We would like to know what is the image of $V \tnsr \bbS_{\alpha}$ under the multiplication. More precisely, let $\bbS_{\eta}$ be a summand of $V \tnsr \bbS_{\alpha}$. We would like to know the image of the following map:

\[
\bbS_{\eta} \into V \tnsr \bbS_{\alpha} \into V \tnsr (\bbS_{\mu} \tnsr \ExtAlg^kV) \xrightarrow{\niso} (V\tnsr \bbS_{\mu})\tnsr \ExtAlg^kV \to \bbS_{\mu+(1,0, \cdots)} \tnsr \ExtAlg^kV
\]

We know the labeled diagram of $\eta$ coming from the bracket placement $V \tnsr (\bbS_{\mu} \tnsr \ExtAlg^k V)$. By the definition of multiplication on $M_{\lambda}$, the image of $\bbS_{\eta}$ is equal to $0$ if $\eta$ is not equivalent to an embedding with "$V$" in the first row. Otherwise, the image is just $\bbS_{\eta}$.

\begin{example}
Let the dimension of $V$ be equal to $2$. Let us work out the lattice of $M_{(1,1)} \tnsr \ExtAlg^1 V$. The lowest degree is $3$ where we find just one representation {\scriptsize\young(\hfil \ExtAlg,\hfil)}. Now, 

\begin{eqnarray*}
{\scriptsize\young(\hfil \ExtAlg,\hfil)} \tnsr {\scriptsize\young(V)} \niso {\scriptsize\young(\hfil \ExtAlg V,\hfil)} \drsum {\scriptsize\young(\hfil \ExtAlg,\hfil V)}
\end{eqnarray*}

But both of these diagrams are equivalent to a diagram with "$V$" in the first row: {\scriptsize\young(\hfil \ExtAlg V,\hfil)} $\sim$ {\scriptsize\young(\hfil V\ExtAlg,\hfil)} and {\scriptsize\young(\hfil \ExtAlg,\hfil V)} $\sim$ {\scriptsize\young(\hfil V,\hfil \ExtAlg)}. Thus in degrees $3$ and $4$ the lattice looks like:

\begin{figure}[h!]

\begin{tikzpicture}[scale=1.4]

\node at (2,0) {{\scriptsize\young(\hfil \ExtAlg,\hfil)}};
\node at (0,1) {{\scriptsize\young(\hfil \hfil \ExtAlg,\hfil)}};
\node at (4,1) {{\scriptsize\young(\hfil \hfil,\hfil \ExtAlg)}};

\path[->,font=\scriptsize] 

(2,0) edge[shorten <= .5cm, shorten >= .5cm] (0,1)
(2,0) edge[shorten <= .5cm, shorten >= .5cm] (4,1); 
\end{tikzpicture}

\caption{Degrees $3$ and $4$ of $M_{(1,1)} \tnsr \ExtAlg^1V$}

\end{figure}
\FloatBarrier

Starting with degree $4$ things stabilize in the sence that each representation $\alpha$ generates a submodule isomorphic to $M_{\alpha}$:

\begin{figure}[h!]

\begin{tikzpicture}[scale=1.4]

\node at (2,0) {{\scriptsize\young(\hfil \ExtAlg,\hfil)}};
\node at (0,1) {{\scriptsize\young(\hfil \hfil \ExtAlg,\hfil)}};
\node at (4,1) {{\scriptsize\young(\hfil \hfil,\hfil \ExtAlg)}};
\node at (0,2) {{\scriptsize\young(\hfil \hfil \hfil \ExtAlg,\hfil)}};
\node at (4,2) {{\scriptsize\young(\hfil \hfil \hfil,\hfil \ExtAlg)}};
\node at (0,3) {$\vdots$};
\node at (4,3) {$\vdots$};

\path[->,font=\scriptsize] 

(2,0) edge[shorten <= .5cm, shorten >= .5cm] (0,1)
(2,0) edge[shorten <= .5cm, shorten >= .5cm] (4,1)
(0,1) edge[shorten <= .5cm, shorten >= .5cm] (0,2)
(0,2) edge[shorten <= .5cm, shorten >= .5cm] (0,3)
(4,1) edge[shorten <= .5cm, shorten >= .5cm] (4,2)
(4,2) edge[shorten <= .5cm, shorten >= .5cm] (4,3);
\end{tikzpicture}

\caption{$M_{(1,1)} \tnsr \ExtAlg^1V$}

\end{figure}
\FloatBarrier

\end{example}

In general, the structure of $M_{\lambda} \tnsr \ExtAlg^kV$ is described by the following 

\begin{lemma}
\label{lem:str}

Suppose first $(j+1), j \geq 1$ rows of $\lambda$ have the same length. Then all the summands $\bbS_{\alpha}$ of $\bbS_{\lambda} \tnsr \ExtAlg^kV$ which have at most $j$  and at least one "$\ExtAlg$"'s in the right-most column will generate two representations in the next degree. All the other summands will generate just one representation. Starting with degree $|\lambda|+k+1$ each representation $\nu$ generates a submodule isomorphic to $M_{\nu}$.

If the first row of $\lambda$ is strictly longer than the second one, then $M_{\lambda} \tnsr \ExtAlg^k V \niso \underset{\alpha \in VS(\lambda,k)}{\Drsum}M_{\alpha}$.
\end{lemma}

\begin{proof}
Let $\bbS_{\alpha} \into \bbS_{\mu} \tnsr \ExtAlg^k V$ be a direct summand and mark the extra boxes of $\alpha$ that were added to $\mu$ with "$\ExtAlg$"'s. If the right-most column  of $\alpha$ consists of $s$ "$\ExtAlg$"'s and the $(s+1)$-st row of $\alpha$ is one box shorter then the row $s$ of $\alpha$ then $\bbS_{\alpha}$ will generate two representations in the next degree: one obtained by adding "$V$" to the first row and then switching it with the "$\ExtAlg$" and the other one by adding "$V$" to the row $(s+1)$  and then sliding it up to the first row. Otherwise, $\bbS_{\alpha}$ will only generate one representation in the next degree - the one obtained from adjoining "$V$" to the first row (and possibly switching it with "$\ExtAlg$" if there is one in the first row). But recall that $\mu$ itself was obtained from $\lambda$ by adjoining boxes to the first row, so nothing starting from the next-to-lowest degree can generate more than one representation. 
\end{proof}

Now we are ready to calculate $\Tor^R_{\bullet}(M_{\lambda}, \Bbbk)$. This is our main result. In particular, it implies that the resolution of $M_{\lambda}$ is linear.

\begin{theorem}
\label{thm:main}
\begin{eqnarray*}
\Tor^R_i(M_{\lambda}, \Bbbk) \niso \underset{\eta \in B}{\Drsum} \bbS_{\eta}
\end{eqnarray*}

where the set $B$ is the set of all representations $\eta$ obtained from $\lambda$ by adding $i$ boxes according to the Pieri rule (for exterior powers), but no box is added to the first row.

\end{theorem}

\begin{proof}

The tautological Koszul complex is a projective resolution of $\Bbbk$:

\begin{figure}[h!]

\begin{tikzpicture}
\matrix(m) [matrix of math nodes, 
row sep=0.5em, column sep=0.5em, 
text height=1.5ex, text depth=0.25ex]
{K^{\bullet} \colon 0 & \underset{-n}{R \underset{\Bbbk}{\tnsr} \ExtAlg^n V} & R \underset{\Bbbk}{\tnsr} \ExtAlg^{n-1} V & \cdots & R \underset{\Bbbk}{\tnsr} \ExtAlg^2 V & R \underset{\Bbbk}{\tnsr} V & \underset{0}{R} & 0\\};

\path[->,font=\scriptsize] 
(m-1-1) edge (m-1-2)
(m-1-2) edge (m-1-3) 
(m-1-3) edge (m-1-4) 
(m-1-4) edge (m-1-5)
(m-1-5) edge (m-1-6)
(m-1-6) edge (m-1-7)
(m-1-7) edge (m-1-8);
\end{tikzpicture}
\end{figure}
\FloatBarrier

After tensoring (over $R$) with $M_{\lambda}$ we get the complex:

\begin{figure}[h!]

\begin{tikzpicture}
\matrix(m) [matrix of math nodes, 
row sep=0.5em, column sep=0.5em, 
text height=1.5ex, text depth=0.25ex]
{M_{\lambda} \underset{R}{\tnsr} K^{\bullet} \colon 0 & \underset{-n}{M_{\lambda} \underset{\Bbbk}{\tnsr} \ExtAlg^n V} & M_{\lambda} \underset{\Bbbk}{\tnsr} \ExtAlg^{n-1} V & \cdots & M_{\lambda} \underset{\Bbbk}{\tnsr} \ExtAlg^2 V & M_{\lambda} \underset{\Bbbk}{\tnsr} V & \underset{0}{M_{\lambda}} & 0\\};

\path[->,font=\scriptsize] 
(m-1-1) edge (m-1-2)
(m-1-2) edge (m-1-3) 
(m-1-3) edge (m-1-4) 
(m-1-4) edge (m-1-5)
(m-1-5) edge (m-1-6)
(m-1-6) edge (m-1-7)
(m-1-7) edge (m-1-8);
\end{tikzpicture}
\end{figure}
\FloatBarrier

The differential $d^{-i} \colon M_{\lambda} \tnsr \ExtAlg^iV \to M_{\lambda} \tnsr \ExtAlg^{i-1}V$ is determined where the lowest degree piece of $M_{\lambda} \tnsr \ExtAlg^iV$ (namely $\bbS_{\lambda} \tnsr \ExtAlg^{i}V$) goes. Its image has to be in $\bbS_{\lambda+(1,0, \cdots)} \tnsr \ExtAlg^{i-1} V$ because the differential is a map of graded modules. 

Consider the direct summand $\bbS_{\alpha}$ of $\bbS_{\lambda} \tnsr \ExtAlg^{i}V$ where $\alpha$ was obtained from $\lambda$ by adding $i$ boxes according to the Pieri rule and one of the boxes is added to the first row. Such a representation also occurs in $\bbS_{\lambda+(1,0, \cdots)} \tnsr \ExtAlg^{i-1} V$ and \emph{nowhere else}. Let us for now assume the following claim: 

\begin{claim}

The differential has to be nonzero when restricted to such representation. 

\end{claim}

The image of the entire submodule generated by $\bbS_{\alpha}$ in $M_{\lambda} \tnsr \ExtAlg^iV$ has to be the submodule generated by $\bbS_{\alpha}$ in $M_{\lambda} \tnsr \ExtAlg^{i-1}V$, which is isomorphic to $M_{\alpha}$ because $d^{-i}$ puts $\bbS_{\alpha}$ in the next-to-lowest degree of $M_{\lambda} \tnsr \ExtAlg^{i-1}V$. Moreover, all the representations in  $M_{\lambda} \tnsr \ExtAlg^iV$ that also appear in $<\bbS_{\alpha}> \subset M_{\lambda} \tnsr \ExtAlg^{i-1}V$ are generated by $\bbS_{\alpha}$ in $M_{\lambda} \tnsr \ExtAlg^iV$.

Now we want to show that nothing from degree $|\lambda|+i+1$ (next-to-lowest degree of $M_{\lambda} \tnsr \ExtAlg^i V$) or higher can contribute to the cohomology $H^{-i}(M_{\lambda} \underset{R}{\tnsr} K^{\bullet})$. From the structure of the modules $M_{\lambda} \tnsr \ExtAlg^k V$ it is clear that it is enough to prove that nothing from degree $|\lambda|+i+1$ contributes to cohomology.

So, consider next-to-lowest degree of $M_{\lambda} \tnsr \ExtAlg^i V$ which is $\bbS_{\lambda+(1,0,\cdots)} \tnsr \ExtAlg^i V$. All the representations in this degree have either one or two more boxes than $\lambda$ in the first row. If there is two more boxes, it means that this summand $\bbS_{\nu}$ was obtained from adding one box to the first row to a summand $\bbS_{\eta}$ of $\bbS_{\lambda} \tnsr \ExtAlg^iV$ which had a box added to the first row of $\lambda$. By the claim $\bbS_{\eta}$ has to map to something nonzero under the differential and because the image of the submodule generated by $\bbS_{\eta}$ in $M_{\lambda} \tnsr \ExtAlg^i V$ has to map to a submodule isomorphic to $M_{\eta}$ we see that differential restricted to $\bbS_{\nu}$ is non-zero, so $\bbS_{\nu}$ will not contribute to cohomology.

If we take a summand $\bbS_{\nu}$ of $\bbS_{\lambda+(1,0,\cdots)} \tnsr \ExtAlg^i V$ such that first row of $\nu$ is longer than the first row of $\lambda$ only by one box then it means that there are representations in the lowest degree of the previous term of the complex which have to map to it, so summands of this kind also contribute nothing to cohomology.

Thus only representations in the lowest degree of $M_{\lambda} \tnsr \ExtAlg^iV$ will contribute to cohomology, and it will be those which are in the kernel of the differential $d^{-i}$ - the ones obtained from $\lambda$ by adding boxes by the Pieri rule with no box added to the first row.

It remains to prove the claim. Assume the contrary. Let us recall a standard argument: given a partition $\alpha$ consider the function which takes finite complex of $A$-modules:

\[
C^{\bullet} \colon \cdots \to C^{i-1} \to C^i \to C^{i+1} \to \cdots
\]

to $\Sigma_i (-1)^i \dim_{\Bbbk} \Hom_{\Bbbk[G]} (\bbS_{\alpha}, C^i)$ - the alternating sum of numbers of copies of $\bbS_{\alpha}$ in the terms of the complex. One can calculate this function either from a complex $C^{\bullet}$ or from its cohomology $HC^{\bullet}$. When we calculate this function from the complex $M_{\lambda} \underset{R}{\tnsr} K^{\bullet}$ itself, we get zero. This means we should also get zero when we calculate it from  $H(M_{\lambda} \underset{R}{\tnsr} K^{\bullet})$. Clearly there is one copy of $\bbS_{\alpha}$ in $H^{-i}(M_{\lambda} \underset{R}{\tnsr} K^{\bullet})$, because it is not in the image of a differential $d^{-i-1}$ by degree considerations, and we assumed that it is in the kernel of the differential $d^{-i} \colon M_{\lambda} \tnsr \ExtAlg^iV \to M_{\lambda} \tnsr \ExtAlg^{i-1}V$. So to cancell this, it must be that $\bbS_{\alpha}$ contributes to the cohomology group $H^{-i+1}(M_{\lambda} \underset{R}{\tnsr} K^{\bullet})$. In particular, the whole submodule generated by $\bbS_{\alpha}$ in $M_{\lambda} \tnsr \ExtAlg^{i-1}V$ is in the kernel of $d^{-i+1}$.  

The entire submodule generated by $\bbS_{\alpha}$ in $M_{\lambda} \tnsr \ExtAlg^iV$ has to go to zero under the differential $d^{-i}$, but this means that nothing hits the submodule generated by $\bbS_{\alpha}$ in $M_{\lambda} \tnsr \ExtAlg^{i-1}V$  (which is isomorphic to $M_{\alpha}$). We see that the entire submodule generated by $\bbS_{\alpha}$ in $M_{\lambda} \tnsr \ExtAlg^{i-1}V$ has to contribute to the cohomology $H^{-i+1}(M_{\lambda} \underset{R}{\tnsr} K^{\bullet})$. This is a contradiction because this group has to be finite-dimensional over $\Bbbk$, because all our modules are finetely generated. 

\end{proof}

\begin{example}
We illustrate the above proof with the complex $M_{(1,1,0)} \underset{R}{\tnsr} K^{\bullet}$

\begin{figure}[H]

\begin{tikzpicture}[scale=1.0]

\node at (0,0) {{\scriptsize\young(\hfil,\hfil)}};
\node at (0,1.5) {{\scriptsize\young(\hfil \hfil,\hfil)}};
\node at (0,3) {$\vdots$};
\node at (0,-1.5) {$M_{(1,1,0)}$};

\node at (-2,3) {{\scriptsize\young(\hfil \hfil \ExtAlg,\hfil)}};
\node at (-2,4.5) {$\vdots$};

\node at (-3,1.5) {{\scriptsize\young(\hfil \ExtAlg,\hfil)}};
\node at (-3,3) {{\scriptsize\young(\hfil \hfil,\hfil \ExtAlg)}};
\node at (-3,4.5) {$\vdots$};

\node at (-4,1.5) {{\scriptsize\young(\hfil,\hfil,\ExtAlg)}};
\node at (-4,3) {{\scriptsize\young(\hfil \hfil,\hfil,\ExtAlg)}};
\node at (-4,4.5) {$\vdots$};
\node at (-3,0) {$M_{(1,1,0) \tnsr V}$};

\node at (-6,4.5) {{\scriptsize\young(\hfil \hfil,\hfil \ExtAlg,\ExtAlg)}};
\node at (-6,6) {$\vdots$};

\node at (-7,3) {{\scriptsize\young(\hfil \ExtAlg,\hfil,\ExtAlg)}};
\node at (-7,4.5) {{\scriptsize\young(\hfil \hfil \ExtAlg,\hfil,\ExtAlg)}};
\node at (-7,6) {$\vdots$};

\node at (-8,3) {{\scriptsize\young(\hfil \ExtAlg,\hfil \ExtAlg)}};
\node at (-8,4.5) {{\scriptsize\young(\hfil \hfil \ExtAlg,\hfil \ExtAlg)}};
\node at (-8,6) {$\vdots$};
\node at (-7,1.5) {$M_{(1,1,0) \tnsr \ExtAlg^2V}$};

\node at (-10,4.5) {{\scriptsize\young(\hfil \ExtAlg,\hfil \ExtAlg,\ExtAlg)}};
\node at (-10,6) {$\vdots$};
\node at (-10,3) {$M_{(1,1,0) \tnsr \ExtAlg^3 V}$};

\path[->,font=\scriptsize] 

(0,0) edge[shorten <= .5cm, shorten >= .0cm] (0,1)

(-2,3) edge[shorten <= .5cm, shorten >= .5cm] (-2,4.5)
(-3,1.5) edge[shorten <= .5cm, shorten >= .5cm] (-3,3)
(-3,1.5) edge[shorten <= .5cm, shorten >= .5cm] (-2,3)
(-3,3) edge[shorten <= .5cm, shorten >= .5cm] (-3,4.5)
(-4,1.5) edge[shorten <= .5cm, shorten >= .5cm] (-4,3)
(-4,3) edge[shorten <= .5cm, shorten >= .5cm] (-4,4.5)

(-6,4.5) edge[shorten <= .5cm, shorten >= .5cm] (-6,6)
(-7,4.5) edge[shorten <= .5cm, shorten >= .5cm] (-7,4.5)
(-7,4.5) edge[shorten <= .5cm, shorten >= .5cm] (-7,6)
(-7,3) edge[shorten <= .5cm, shorten >= .5cm] (-6,4.5)

(-8,3) edge[shorten <= .5cm, shorten >= .5cm] (-8,4.5)
(-8,4.5) edge[shorten <= .5cm, shorten >= .5cm] (-8,6)

(-10,4.5) edge[shorten <= .5cm, shorten >= .5cm] (-10,6)

(-3,1.5) edge[shorten <= .5cm, shorten >= .5cm] (0,1.5)
(-8,3) edge[shorten <= .5cm, shorten >= .5cm] (-3,2.8)
(-7,3) edge[shorten <= .5cm, shorten >= .5cm] (-4,3.2)
(-10,4.5) edge[shorten <= .5cm, shorten >= .5cm] (-6,4.5)

(-7,3) edge[shorten <= .5cm, shorten >= .5cm] (-7,4.5);

\end{tikzpicture}

\caption{$M_{(1,1,0)} \tnsr K^{\bullet}$}
\end{figure}
\FloatBarrier

As we see from the diagram the there are only two cohomology groups, namely $H^0(M_{(1,1,0)} \underset{R}{\tnsr} K^{\bullet})=\scriptsize\young(\hfil,\hfil)$ and $H^1(M_{(1,1,0)} \underset{R}{\tnsr} K^{\bullet})=\scriptsize\young(\hfil,\hfil,\hfil)$ as it should be by the above theorem.
 
\end{example}

\begin{example}
The tautological Koszul complex and its cut-offs which resolve the modules of cycles is a familiar example of resolutions for $M_{\lambda}$ where $\lambda$ is a partition corresponding to an exterior power.
\end{example}

\begin{example}
Note that our results agree with the known resolutions of powers of the maximal ideal generated by $V$ in $R$. See Remark 2.2 in \cite{Guard}.
\end{example}

\begin{example}
Let us continue with $R \tnsr \Sym_2 V$, where dimension of $V$ is $3$. We saw that $<\bbS_{(2,2,0)}> \niso M_{(2,2,0)}$. The resolution of this module should look as follows:

\begin{figure}[h!]

\begin{tikzpicture}
\matrix(m) [matrix of math nodes, 
row sep=0.5em, column sep=0.5em, 
text height=1.5ex, text depth=0.25ex]
{0 & \underset{-1} {R \tnsr {\scriptsize\young(\hfil \hfil,\hfil \hfil,\hfil)}} & \underset{0}{R \tnsr {\scriptsize\young(\hfil \hfil,\hfil \hfil)}} & 0\\};

\path[->,font=\scriptsize] 
(m-1-1) edge (m-1-2)
(m-1-2) edge (m-1-3)
(m-1-3) edge (m-1-4);
\end{tikzpicture}
\end{figure}

and indeed this is the answer produced by Macaulay 2.
\end{example}

\begin{example}
Now let us look at $(R \tnsr \Sym_2 V)/<\bbS_{(2,2,0)}>$ (dimension of $V$ is still $3$). We saw that the submodule generated by $\bbS_{(2,1,0)}$ in this quotient is isomorphic to $M_{(2,1,0)}$. So the resolution is

\begin{figure}[h!]

\begin{tikzpicture}
\matrix(m) [matrix of math nodes, 
row sep=0.5em, column sep=0.5em, 
text height=1.5ex, text depth=0.25ex]
{0 & \underset{-2}{R \tnsr {\scriptsize\young(\hfil \hfil,\hfil \hfil,\hfil)}} & R \tnsr ({\scriptsize\young(\hfil \hfil,\hfil \hfil)} \Drsum {\scriptsize\young(\hfil \hfil,\hfil,\hfil)}) & \underset{0}{R \tnsr {\scriptsize\young(\hfil \hfil,\hfil)}} & 0\\};

\path[->,font=\scriptsize] 
(m-1-1) edge (m-1-2)
(m-1-2) edge (m-1-3)
(m-1-3) edge (m-1-4)
(m-1-4) edge (m-1-5);
\end{tikzpicture}
\end{figure}

and indeed this is the answer produced by Macaulay 2.
\end{example}

\subsection{Using the Geometric Method}
\label{subsec:Geom}

One can use the geometric method (Theorem 5.1.2 of \cite{Wey}) to prove our Theorem~\ref{thm:main}. Consider the projective space $\bbP:=\Proj R = \Proj \Sym V$. We have the (exact) tautological sequence of vector bundles:

\[
0 \to \scrptR \to V^* \times \bbP \to \scrptQ \to 0
\]

Where $\scrptR$ is the tautological rank $1$ subbundle of the trivial bundle $\scrptR:=\{(v,[L]) \in V^* \times \bbP | v \in L\}$.

Given a partition $\lambda=(\lambda_1, \lambda_2, \cdots \lambda_n)$, let $\gamma$ be the partition $\gamma=(\lambda_2, \lambda_3, \cdots \lambda_n)$. Define the vector bundle $\scrptV:=(\Sym_{\lambda_1} \scrptR^*) \tnsr \bbS_{\gamma} \scrptQ^*$. Now Theorem 5.1.2 of \cite{Wey} tells us how to construct the complex $F(\scrptV)_{\bullet}$ which in our case will be the minimal free resolution of the graded $R$-module $H^0(\bbP, \Sym(\scrptR^*) \tnsr \scrptV)$. In degree $d$ of $H^0(\bbP, \Sym \scrptR^* \tnsr \scrptV)$ we have $H^0(\bbP, \Sym_d \scrptR^* \tnsr \Sym_{\lambda_1} \scrptR^* \tnsr \bbS_{\gamma} \scrptQ^*) \niso H^0(\bbP, \Sym_{\lambda_1+d} \scrptR^* \tnsr \bbS_{\gamma} \scrptQ^*)$ because $\scrptR^*$ is locally free of rank $1$. Using Corollary 4.1.9. of \cite{Wey} this cohomology is just $\bbS_{\lambda+(d,0,\cdots)}V$. So the complex $F(\scrptV)_{\bullet}$ will resolve $M_{\lambda}$.

The terms of the complex $F(\scrptV)_{\bullet}$ are

\[
F(\scrptV)_i= R \tnsr \underset{j \geq 0}{\Drsum} H^j(\bbP, \ExtAlg^{i+j} \scrptQ^* \tnsr \scrptV)=R \tnsr \underset{j \geq 0}{\Drsum} H^j(\bbP, \Sym_{\lambda_1} \scrptR^* \tnsr (\ExtAlg^{i+j} \scrptQ^* \tnsr \bbS_{\gamma} \scrptQ^*))
\]

Again, using Corollary 4.1.9 and Bott's Algorithm 4.1.5 of \cite{Wey} we see that we get the same answer as in our Theorem~\ref{thm:main}.

\section{Resolutions of Truncations of $M_{\lambda}$}
\label{sec:ResTr}

By examining the proof of Theorem~\ref{thm:main} we see that it should also work for truncations of $M_{\lambda}$'s, but some care needs to be taken for what happens at the top of each term in the complex.  There will be two contributions to cohomology: one from the top and one from the bottom. Hopefully this explains the choice of letters "$T$" and "$B$" below.

\begin{theorem}
\label{thm:main2}
Suppose $l \geq 2$, then

\[
\Tor^R_i(M_{\lambda}/V^l M_{\lambda}, \Bbbk) \niso (\underset{\eta \in B}{\Drsum} \bbS_{\eta}) \drsum (\underset{\eta \in T}{\Drsum} \bbS_{\eta})
\]

where the set $B$ is the set of all representations $\eta$ obtained from $\lambda$ by adding $i$ boxes according to the Pieri rule (for exterior powers), but no box is added to the first row. The set $T$ is the set of all representations obtained from $\lambda+(l-1, 0, \cdots)$ by adding $i$ boxes according to Pieri rule (for exterior powers) and one box is added to the first row. Equivalently, the set $T$ consists of all representations obtained from $\lambda+(l,0, \cdots)$ by adding $i-1$ boxes according to Pieri rule for exterior powers, but no box is added to the first row.  

When $l=1$:

\[
\Tor^R_i(M_{\lambda}/V M_{\lambda}, \Bbbk) \niso \underset{\eta \in VS(\lambda, i)}{\Drsum} \bbS_{\eta}
\]

\end{theorem}

We will find the resolution of $M_{\lambda}/V M_{\lambda}$ in a different way in the next section.

\begin{proof}

We will use the complex $(M_{\lambda}/V^l M_{\lambda}) \underset{R}{\tnsr} K^{\bullet}$ to compute the $\Tor$.

First, the case $l=1$. By degree reasons the differentials in the complex are all $0$ and the result follows immediately.

Now let $l \geq 2$. Let us find $H^{-i}(M_{\lambda}/V^l M_{\lambda} \underset{R}{\tnsr} K^{\bullet})$ All the arguments from the proof of Theorem~\ref{thm:main} remain valid, so the contribution $\underset{\eta \in B}{\Drsum} \bbS_{\eta}$ is definitely there.

The top of the term $M_{\lambda}/V^l M_{\lambda} \tnsr \ExtAlg^iV$ has nowhere to go under the differential for degree reasons, so now it remains to see which of the summands of the top of this term are in the image of the differential. In fact, this is determined by what happens in the next-to-lowest degree of $M_{\lambda}/V^l M_{\lambda} \tnsr \ExtAlg^iV$. Each representation at the top of $M_{\lambda}/V^l M_{\lambda} \tnsr \ExtAlg^i$ is obtained from some representation in the next-to-the lowest degree by adjoining $l-2$ boxes to the first row. From the proof of Theorem~\ref{thm:main} we know that each representation in the next-to-lowest degree which has the first row longer than the first row of $\lambda$ just by one box are in the image. It means that all the representations at the top which were generated by these are also in the image, thus they do not contribute to cohomology. Thus the only representations at the top that do contribute to cohomology are the ones generated by a representation in next-to-lowest degree whose first row is longer than the first row of $\lambda$ by two boxes. These are exactly the representations from $T$.

\end{proof}
 
\section{Simple Objects}
\label{sec:SimOb}

The simple objects in $A-\fmod$ are the irreducible representations $\bbS_{\lambda}$ where all the variables act by zero.

Consider (again!) the tautological Koszul complex $K^{\bullet}$, which is the projective resolution of $\Bbbk=\bbS_{\emptyset}$:

\begin{figure}[h!]

\begin{tikzpicture}
\matrix(m) [matrix of math nodes, 
row sep=0.5em, column sep=0.5em, 
text height=1.5ex, text depth=0.25ex]
{K^{\bullet} \colon 0 & \underset{-n}{R \underset{\Bbbk}{\tnsr} \ExtAlg^n V} & R \underset{\Bbbk}{\tnsr} \ExtAlg^{n-1} V & \cdots & R \underset{\Bbbk}{\tnsr} \ExtAlg^2 V & R \underset{\Bbbk}{\tnsr} V & \underset{0}{R} & 0\\};

\path[->,font=\scriptsize] 
(m-1-1) edge (m-1-2)
(m-1-2) edge (m-1-3) 
(m-1-3) edge (m-1-4) 
(m-1-4) edge (m-1-5)
(m-1-5) edge (m-1-6)
(m-1-6) edge (m-1-7)
(m-1-7) edge (m-1-8);
\end{tikzpicture}
\end{figure}

If we now apply the exact functor $\_\underset{\Bbbk}{\tnsr} \bbS_{\lambda}$ we will get a projective resolution of the simple module $\bbS_{\lambda}$:

\begin{figure}[h!]

\begin{tikzpicture}
\matrix(m) [matrix of math nodes, 
row sep=0.5em, column sep=0.5em, 
text height=1.5ex, text depth=0.25ex]
{K^{\bullet}\underset{\Bbbk}{\tnsr}\bbS_{\lambda} \colon 0 & \underset{-n}{R \underset{\Bbbk}{\tnsr} 
(\ExtAlg^n V\underset{\Bbbk}{\tnsr}\bbS_{\lambda})} & R \underset{\Bbbk}{\tnsr} (\ExtAlg^{n-1} V \underset{\Bbbk}{\tnsr}\bbS_{\lambda}) & \cdots & R \underset{\Bbbk}{\tnsr} (\ExtAlg^2 V\underset{\Bbbk}{\tnsr}\bbS_{\lambda}) & R \underset{\Bbbk}{\tnsr} (V \underset{\Bbbk}{\tnsr}\bbS_{\lambda}) & \underset{0}{R\underset{\Bbbk}{\tnsr}(\bbS_{\lambda}}) & 0\\};

\path[->,font=\scriptsize] 
(m-1-1) edge (m-1-2)
(m-1-2) edge (m-1-3) 
(m-1-3) edge (m-1-4) 
(m-1-4) edge (m-1-5)
(m-1-5) edge (m-1-6)
(m-1-6) edge (m-1-7)
(m-1-7) edge (m-1-8);
\end{tikzpicture}
\end{figure}

Note that the terms in this resolution are the same as predicted by Theorem~\ref{thm:main2}. By degree considerations, when we apply $\Hom_A(\_, \bbS_{\eta})$ to the above resolution, all the differentials become $0$. Now we can calculate $\Ext_{A}^{\bullet}(\bbS_{\lambda}, \bbS_{\eta})$:

\begin{eqnarray*}
\Ext_{A}^i(\bbS_{\lambda}, \bbS_{\eta})=\Hom_A(R\underset{\Bbbk}{\tnsr} (\ExtAlg^i V \underset{\Bbbk}{\tnsr} \bbS_{\lambda}), \bbS_{\eta}) \niso \Hom_R(R\underset{\Bbbk}{\tnsr} (\ExtAlg^i V \underset{\Bbbk}{\tnsr} \bbS_{\lambda}), \bbS_{\eta})^G \niso  \\
\Hom_{\Bbbk}(\ExtAlg^i V \underset{\Bbbk}{\tnsr} \bbS_{\lambda}, \bbS_{\eta})^G \niso
\begin{cases}
\Bbbk \text{ if $\bbS_{\eta}$ occurs in $\ExtAlg^i V \underset{\Bbbk}{\tnsr} \bbS_{\lambda}$} \\
0 \text{ otherwise }
\end{cases}
\end{eqnarray*}

One can imagine the $1$-extension which corresponds to the nonzero element of $\Ext_{A}^1(\bbS_{\lambda}, \bbS_{\eta})$. For example, when $\eta=\scriptsize\young(\hfil \hfil \hfil,\hfil \hfil)$ and $\lambda=\scriptsize\young(\hfil \hfil \hfil,\hfil)$ ($\eta$ has one box added to the second row of $\lambda$) the $1$-extension is the complex:

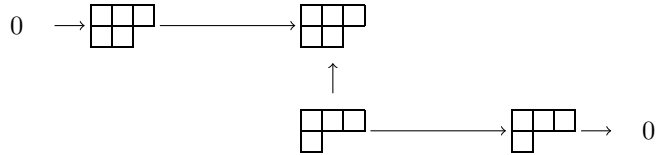
\begin{figure}[h!]

\begin{tikzpicture}[scale=1.4]

\node at (-5,1) {$0$};
\node at (-4,1) {{\scriptsize\young(\hfil \hfil \hfil,\hfil \hfil)}};
\node at (-2,1) {{\scriptsize\young(\hfil \hfil \hfil,\hfil \hfil)}};
\node at (-2,0) {{\scriptsize\young(\hfil \hfil \hfil,\hfil)}};
\node at (0,0) {{\scriptsize\young(\hfil \hfil \hfil,\hfil)}};
\node at (1,0) {$0$};
\path[->,font=\scriptsize] 

(-5,1) edge[shorten <= .5cm, shorten >= .5cm] (-4,1)
(-4,1) edge[shorten <= .5cm, shorten >= .5cm] (-2,1)
(-2,0) edge[shorten <= .5cm, shorten >= .5cm] (-2,1) 
(-2,0) edge[shorten <= .5cm, shorten >= .5cm] (0,0) 
(0,0) edge[shorten <= .5cm, shorten >= .5cm] (1,0);

\end{tikzpicture}

\end{figure}
\FloatBarrier

where in the middle we have the lattice of a small equivariant module which only has two representations in it, the map on the left is the inclusion and map on the right is the projection.

\end{document}